%% file: main.tex
\documentclass[reqno]{amsart}
\usepackage{amsmath, amssymb, amsthm, epsfig}
\usepackage{hyperref, latexsym}
\usepackage{url}
\usepackage[mathscr]{euscript}
\usepackage{amsthm} 
\usepackage{bbm}

\usepackage{color}
\usepackage{fullpage} 
\usepackage{setspace}
\usepackage{comment}

\usepackage{array,booktabs}

\def\today{\ifcase\month\or
  January\or February\or March\or April\or May\or June\or 
  July\or August\or September\or October\or November\or December\fi
  \space\number\day, \number\year}

 \newtheorem{theorem}{Theorem}[section]
  
 \newtheorem{lemma}[theorem]{Lemma}
 \newtheorem{proposition}[theorem]{Proposition}

 \newtheorem{defi}[theorem]{Definition}

 \newtheorem{remark}[theorem]{Remark}

 \newcommand{\mc}{\mathcal}
 \newcommand{\A}{\mc{A}}

 \newcommand{\M}{\mc{M}}

 \newcommand{\R}{\mathbb{R}}
  
 \newcommand{\N}{\mathbb{N}}
 
 \newcommand{\Z}{\mathbb{Z}}

 \newcommand{\p}{\mathbbm{p}}
 \newcommand{\Pe}{\mathbbm{P}}

    \renewcommand{\d}{\text{\rm d}}

\newcommand{\var}{{\rm Var\,}}

\newcommand{\letf}{$f: \R \rightarrow \R^+_0$}
\newcommand{\peak}{$\{ p<r<q \}$}
\newcommand{\ma}{M^{\alpha}}

\begin{document}

\title[]{Study of bounded variation on the H-L nontangential operator}
\author[Fred]{Frederico Toulson}
\date{\today}
\subjclass[2010]{42B25, 26A45, 46E35, 46E39.}
\keywords{Hardy-Littlewood maximal operator, bounded variation.}

\address{
IST - Instituto Superior Técnico, Av. Rovisco Pais, Nº 1 1049-001 Lisboa, Portugal.
address}

\email{email fredericotoulson@tecnico.ulisboa.pt
}

\allowdisplaybreaks
\numberwithin{equation}{section}

\begin{abstract}

In this paper we present a proof of sharp boundedness of the discrete 1-dimensional Hardy-Littlewood nontangential maximal operator, when the parameter is in the range $[\frac{1}{3},+\infty)$. This generalizes a theorem by Bober, Carneiro, Hughes and Pierce, where they prove the same result for the uncentered version of the maximal operator. We also use analogous ideas to give an alternative proof for the continuous version of the theorem, by Ramos.
\end{abstract}

\maketitle

\section{Introduction}
\input{intro.tex}

\section{The maximal function and Peaks}

\input{cont.tex}

\section{Discrete case}

\input{disc.tex}

\section{General case}

\input{general}

\subsection*{Acknowledgments} The author acknowledges the support of the Fundação para a Ciência e a Tecnologia (FCT) through the LISMATH program PhD grant PD/BD/150367/2019 and by FCT/Portugal through project UIDB/04459/2020 with DOI identifier 10-54499/UIDP/04459/2020.

\end{document}

%% file: intro.tex
In the study of harmonic analysis, the Hardy-Littlewood maximal operator is a central tool in understanding the behaviour of functions, as well as in many results regarding pointwise convergence.

\begin{defi}[Hardy-Littlewood Maximal Operator]

    Let $f\in L^1_{loc}(\R ^n)$, we define the centered Hardy-Littlewood maximal operator as:
    $$
    Mf(x)= \sup_{r>0} |B_r(x)|^{-1}\int_{B_r(x)} |f(y)|dy,
    $$
where $B_r(x)$ is a ball centered at $x$ with radius $r$, $|B_r(x)|$ is the measure of the ball, and the supremum is taken over all $r>0$.
\end{defi}

We can also define an uncentered version of the same operator, $\widetilde{M}$. The only difference in the definition is that, while the balls in the centered version are centered at  $x$, in the uncentered version the supremum is taken over all balls containing $x$.


When studying an averaging operator it is expected that, in some sense, it will smooth the function to which it is applied. Kinnunen initiated the study of results exploring this expectation \cite{Ki}, proving that the operator is bounded on Sobolev spaces $W^{1,p}$ for $ p > 1 $, for both the centered and uncentered versions. Since the operator does not map $L^1$ into $L^1$- more precisely the image of any non-trivial function is never in $L^1$- the result can't be directly generalized down to the endpoint $p = 1$. As a natural alternative to studying the smoothness at the endpoint of $[1,+\infty )$ the question of studying the operator just at the derivative level was considered. Hajlasz and Onninen posed this general question in \cite{HO}: Given $f\in W^{1,1}(\R ^d)$, is $M f$ weakly differentiable? If so, can the $L^1$ norm of the derivative of $M f$ be bounded by the $W^{1,1}$ norm of $f$?

The general answer to these questions is still unknown, but there are partial results. For further exploration of variants of the maximal operator, particularly concerning boundedness and continuity at the derivative level, Carneiro's survey \cite{Ca} has a compilation of results until 2019, to sugest a few like \cite{AP}, \cite{CH}, \cite{CM}, \cite{CMa}, \cite{CMP}, \cite{KiSa}, \cite{Ko}, \cite{LW}, \cite{PPSS}. Other recent relevant results can be seen at \cite{BW},\cite{GR}.

From now on, we are going to restrict our study to the one dimensional case, $d=1$. We can still consider both, the centered, or the uncentered operator, which have very different properties in this context.

The uncentered version has the property that the operator has no maxima in the detachment set. This set refers to the points where the function and its maximal function differ, given an absolutely continuous representative in the equivalence class of $f\in W^{1,1}(\R)$. This property will play a fundamental role in the proofs we will present.

In one dimension, the first results for the uncentered maximal operator were obtained by Tanaka \cite{Ta}. His work showed that for $f\in W^{1,1}(\R)$, $\widetilde{M}f$ is absolutely continuous, and that the $L^1$ norm of the derivative of $\widetilde M f$ can be bounded by twice the $W^{1,1}$ norm of $f$.

A stronger and more general version of this result was later obtained by Aldaz and Pérez Lázaro in \cite{AL}. They were the first to achieve a result for the space of functions with bounded variation. Here $\var (f)$ denotes the total variation of $f$, which generalizes the $L^1$ norm of the derivative for functions that do not have a derivative in a classical sense. They established that, for a function of bounded variation $f$, 
$$
\var(\widetilde M f) \le \var(f).
$$
The same result was also conjectured for the the centered maximal operator. However, it was only several years later, in 2015, that Kurka\cite{Ku} was able to prove the boundedness of the centered maximal operator, albeit with a constant significantly greater than the conjectured $1$.

In 2017, Ramos \cite{JPR} formulated a generalized version of the maximal operator, which contains as particular cases the centered and uncentered operators. 
\begin{defi}[Hardy-Littlewood Nontangential Maximal Operator]
    Given $f\in BV(\R)$, for $\alpha >0$,
    $$ 
    M^{\alpha}f(x) = \sup_{|x-y| \le \alpha t} \frac{1}{2t} \int_{y-t}^{y+t} |f(s)| \d s, $$
    where the supremum is taken over all $t>0$, and $y\in \R$ verifying the condition.
\end{defi}
The general question can now be adapted to this more general operator as: For which values of $\alpha$ do we have that $\var(M^{\alpha}f) \le \var(f)$? 

The best result, due to Ramos \cite{JPR}, is the following:

\begin{theorem}\label{JPGR} Fix $f \in \textit{BV}(\R)$. For every $\alpha \in [\frac{1}{3},+\infty),$ we have that:
\begin{equation}\label{maint}
\var(M^{\alpha}f) \le \var(f).
\end{equation}  
\end{theorem}
In this paper, we give an alternative, more elementary proof of this result, which will help in understanding where the endpoint $\alpha = \frac 13$ comes from.

Beyond the one-dimensional result of Ramos, multidimensional generalizations of maximal inequalities have also been proven by Carneiro, Gonz{\'a}lez-Riquelme, and Madrid \cite{CGRM}. 
One can also ask which results hold for discrete versions of the same problems, that is, for functions with domain $\Z$. 
\begin{defi}[Hardy-Littlewood Discrete Maximal Operator]

    Let $f: \Z \rightarrow \R$. We define the centered discrete Hardy-Littlewood maximal operator as:
    \begin{equation}\label{HLDCMO}
        \M f(n)= \sup_{r\geq 0} \frac{1}{2r+1}\sum_{i=n-r}^{n+r} |f(i)|,       
    \end{equation}

where the supremum is taken over $r\in \N$.
\end{defi}

In the discrete setting, we have a standard definition of total variation:

\begin{defi}[Total variation of a discrete function]
    Given $f:\Z \rightarrow \R$, we define the total variation over a finite discrete interval $I=[a..b]$, $\var_I (f)$ as:

    $$\var_I(f) = \sum_{i=a}^{b-1}|f(i+1)-f(i)|.$$

    The total variation of the function $\var (f)$ can be taken as the supremum of the variations over all finite intervals.

\end{defi}

\begin{defi}[Bounded variation]

    The class of functions with bounded variation in $\Z$ is defined as:
    
    $$
    BV(\Z)= \{ f:\Z \rightarrow \R |\  \var f < \infty \}.
    $$

    \label{BVar}
\end{defi}

The boundedness of the centered discrete maximal operator \ref{HLDCMO} in $BV(\Z)$ was recently proven by Temur \cite{Te}. His proof was an adaptation of the continuous case proof, by Kurka, and the boundedness constant is still far from the conjectured best value of $1$. 

For the uncentered version of this operator, Bober, Carneiro, Pierce and Hughes  \cite{BCHP} proved that the total variation of the maximal function can be bounded by the total variation of the original function. This is the result we aim to generalize in this paper. By adapting our new proof for the nontangential maximal operator, from the continuous to the discrete case, we can derive a new theorem for the following discrete analog of the operator.

\begin{defi}[Hardy-Littlewood Discrete Nontangential Maximal Operator]

    Let $f: \Z \rightarrow \R^+_0$, we define the Hardy-Littlewood discrete nontangential maximal operator as:
    $$
    \M^\alpha f(x)= \sup_{j \in \N, y \in \Z, |x-y| \leq \lceil \alpha R \rceil } \frac 1{2R+1} \sum_{i=y-R}^{y+R} f(i).
    $$
    
    \label{H-L}
\end{defi}

\begin{theorem}\label{diangle} Fix any $f \in BV(\Z).$ For every $\alpha \in [\frac{1}{3},+\infty),$ we have that

    \begin{equation}\label{dimaint}
    \var(\M^{\alpha}f) \le \var(f).
    \end{equation}
    
\end{theorem}

From now on, without loss of generality, all proofs will be done assuming $f(x) \geq 0$. We can assume this because if we have a bound for non-negative functions, we have $$ \var Mf = \var M|f| \lesssim \var |f| \leq \var f$$.

In the second section of the paper we will recall some usefull facts about the Hardy-Littlewood maximal operator, describe which values the maximal function can take, and some behaviours of the maximal function in the case where $f$ is of finite variation. We will also present some tools which will help us formulate our proof, as well as some preliminary results.

The third section of the paper will be dedicated to proving the new result. We will build upon the ideas presented in the second section and adapt them for the discrete case. This entails establishing a discrete definition of the Hardy-Littlewood Nontangential Maximal Operator. Careful consideration will be given to selecting a definition that generalizes the uncentered maximal operators, with the aim of deriving the most comprehensive and general result possible.

In the last section of the paper, we will present the general proof of \ref{maint}. It is a more tecnical result which gives a more elementary proof than the original, by Ramos.

%% file: cont.tex
In this section, we want to provide some tools to control localy the variation of a maximal function by the variation of the original $f$. These tools are useful in both the continuous and discrete settings. We will also explore some facts about the maximal function.

For the Hardy-Littlewood maximal operator, the particular cases of the centered and uncentered operators, we can characterize the possibilities of values the maximal function can take in a given point. This characterization will help us present our proof.

\begin{enumerate}
    \item\label{casos_maximal} For the continuous version of the operator: Recall that the average of $f\in L^1$, on an interval centered at $y$ with width $2t$,  $\frac{1}{2t} \int_{y-t}^{y+t} |f(s)| \d s$ is continuous in $y $ and in $t\ne 0$, so the supremum in the definition of the maximal function is obtained either at a fixed $(y,t)$ or at the limit superior when $t\rightarrow \{0, \infty \}$. 
    When the operator is discrete, the supremum is attained either as the averaged of a fixed interval, or as the limit superior when the radius goes to infinity.
    \item\label{inf=min} Note that, any choice of function in $BV$ has limits to plus and minus infinity, call them $A$ and $B$, because the function has bounded variation. This means, the limit superior of the averages when the radius goes to infinity is $\min (Mf)=\frac{A+B}{2}$ for the centered version of the operator. For other versions, the minimum $\min (Mf)$ will be a weighted average between $A$ and $B$, but will not depend on $x$.
    We can then know that at any choosen point $x$, the maximal function is at least $\min (Mf)$. If we get the case of $Mf(y)> Mf(x)$, we can conclude that $Mf(y)$ is obtained either when the radius is finite, or as the limit superior when the radius goes to zero.
\end{enumerate}

We will now prove a small lemma which will be usefull in the continuous case, to find a point where the function is below the maximal function. In the discrete case, this is not necessary, as the definition implies $f(x)\leq \M f(x)$

\begin{lemma}
For all $p\in \R$, and for all open sets $U \ni p$, there exists an $a\in U$ such that $f(a) \leq M^{\alpha}f(p)$
\label{smaller}
\end{lemma}
    
\begin{proof}
To prove this claim, we note that if $U$ is a neighbourhood of $p$, then there is an interval $I \subset U$ centered on $p$. Then, $M^{\alpha}f(p) \geq \frac{1}{|I|}\int_I |f(s)| \d s$. If the average of the function on $I$ is smaller or equal than the maximal function, then the interval must have at least one point where the function is also smaller than or equal to $M^{\alpha}f(p)$.
\end{proof}

If, for any type of maximal operator $M$, we can control the variation of $Mf$ within any interval $I$ by the variation of $f$, it follows that $\var Mf \leq \var f$. In \cite{Ku}, Kurka presents us with the notion of peaks, which can be used to partition $\R$ (or $\Z$) into intervals where this control becomes easier to analyze.

\DeclareEmphSequence{\bfseries\itshape}

\begin{defi} \label{peak}
    Let $f \in L^1_{loc}(\R)$ or be a discrete one-dimensional function. We define:
    \begin{itemize}
    \item A \emph{peak} $\p$ for $f$ is a set consisting of three points \peak such that $ f(p) < f(r) > f(q) $,
    \item The \emph{variation} of the function on a peak $ \p = \{ p < r < q \} $ is given by
    $$ \var_\p(f) = 2f(r) - f(p) - f(q), $$
    \item A \emph{system of peaks} for $f$ is a set, $\Pe = \{ p_1 < r_1 < q_1 = p_2 < \cdots < r_n < q_n \}$ such that, for all $1 \leq i \leq n$, $p_i < r_i < q_i$ is a peak.
    \item The \emph{variation} of the function over a system $ \mathbbm{P} $ of peaks for $f$ is
    $$ \var_\mathbbm{P}(f) = \sum _{\p \in \mathbbm{P}} \var_\p(f), $$
    \end{itemize}
\end{defi}

Systems of peaks are a valuable tool in controlling the total variation of a maximal function. Given a finite subset of the function's domain, the variation of the maximal function on this set can be bounded by the variation within a system of peaks, plus the variation outside the system.

To clarify this idea, recall that for any function $f$, the value of $\var(f)$ is the supremum of the variation of $f$ over any finite subset $S = \{ x_1<x_2<...<x_N \}$ of the domain. The variation within the subset can be decomposed into the variation over a system of peaks, plus the decreasing variation at the start, and the increasing variation at the end.

\begin{proposition}\label{divide_int}
    Let $f \in L^1_{loc}(\R)$ or be a discrete one-dimensional function and let $S=\{ x_1<x_2<...<x_N \}$ be a finite subset of the respective domain such that the maximal function is non-monotonous in $S$.
    Let $k$ be the smallest index such that $Mf(x_k) < Mf(x_{k+1}) $, and $j$ the largest with $Mf(x_j) < Mf(x_{j-1}) $. 
    If $k<j$
    Then, we can find a system of peaks $\mathbbm{P}\subset \{ x_k<...<x_j \}$ such that:
    $$
    \sum_{i=1}^{N-1} |Mf(x_{i+1})-Mf(x_i)| \leq Mf(x_1)-Mf(x_k)+ \var_\mathbbm{P}(Mf) + Mf(x_N)-Mf(x_j).
    $$
    In particular, if $Mf(x_1)<Mf(x_2)$ and $Mf(x_{N-1})>Mf(x_N)$,
    $$
    \sum_{i=1}^{N-1} |Mf(x_{i+1})-Mf(x_i)| \leq \var_\mathbbm{P}(Mf).
    $$
    If $k\geq j$:
    $$
    \sum_{i=1}^{N-1} |Mf(x_{i+1})-Mf(x_i)| \leq Mf(x_1)-Mf(x_j) + Mf(x_N)-Mf(x_k).
    $$
    
\end{proposition}

\begin{proof}
    
    We can divide the sum $\sum_{i=1}^{N-1} |Mf(x_{i+1})-Mf(x_i)|$ into three parts, where the first and last of them might be empty if $k=1$ or $j=N$, and the middle one is empty if $k\geq j$. The parts are as follows: the first part is $\sum_{i=1}^{k-1} |Mf(x_{i+1})-Mf(x_i)|$, the second part is $\sum_{i=k}^{j-1} |Mf(x_{i+1})-Mf(x_i)|$, and the third part is $\sum_{i=j}^{N-1} |Mf(x_{i+1})-Mf(x_i)|$.
    For the first and the last parts, we can use the monotony of the maximal function. Since the maximal function is monotonic in the chosen points, we have that $\sum_{i=1}^{k-1} |Mf(x_{i+1})-Mf(x_i)|= Mf(x_1)-Mf(x_k)$ and similarly, $\sum_{i=j}^{N-1} |Mf(x_{i+1})-Mf(x_i)|=Mf(x_N)-Mf(x_j)$.
    For the middle sum, if it exists, we can construct a system of peaks with the same variation as follows: Start with $p_1=x_k$. Then, choose $m>1$ as the smallest index such that $Mf(x_{m+1})<Mf(x_m)$, and $n>m$ as the smallest index such that $Mf(x_{n+1})>Mf(x_n)$. Define $r_1=x_m$ and $q_1=x_n$. If $q_1< x_j$, we continue the process by setting $p_2=q_1$ and repeating the first step, now for the second peak. The process repeats until some $q_i = x_j$, at which point it stops.
    When $k\geq j$, $Mf(x_j) = Mf(x_k)$ and we can bound by the first and last parts.

\end{proof}

The goal now is to prove that if we can control the variation of the maximal function using its variation within a system of peaks, along with its variation outside this system, then we need to control the variation of the maximal function by the variation of the function itself in both cases. Controlling the variation within a system of peaks will depend on the chosen operator. However, our aim is to argue that for all maximal functions we have defined, both in the continuous and discrete settings, given a fixed $S=\{ x_1<x_2<...<x_N \}$, we can control the variation outside a system of peaks (chosen as in the previous proposition), by the variation of the function outside the interval containing the system of peaks.

\begin{proposition}\label{extremos}
    Let $f \in L^1_{loc}(\R)$ or be  a discrete one-dimensional function, and let $x<y$ be two real numbers. Assume $Mf(x)\leq Mf(y)$. Then, $\forall \epsilon > 0, \ \forall \delta >0 $, we can bound $Mf(y)-Mf(x)$ by $|f(w)-f(z)|+\delta$, with $w>z$, for well chosen $w\in (x,\infty)$, and $z \in (x-\epsilon,\infty)$. For the discrete case, make the choice in the intersection of each of these intervals with the integers.
    The analogous result holds for $Mf(x)\geq Mf(y)$, and for $y<x$.
\end{proposition}
\begin{proof}
    If $Mf(x) = Mf(y)$, the proof is complete by choosing, for example, $w=y$, and $z=x$.
 
    To show that we can select a $w$ such that $f(w)\geq Mf(y)-\delta$, proceed by contradiction. Assume no such $w$ exists, meaning that $\forall a>x$ we have that $f(a)<Mf(y)-\delta$. Since $Mf(y)> Mf(x)$, the value of the maximal function at $y$ is either the average over a finite interval, call it $I$, or the limit of averages of intervals whose lengths tend to zero.

    For the case where we have a fixed interval $I$, under our assumption, $I$ must be large enough to contain at least one point outside the interval $(x,\infty)$. This is true because the average of values strictly below $Mf(y)$ cannot be $Mf(y)$. Hence, $x\in I$. 

    Next, we can select an interval $J$ around $x$, starting at the beginning of $I$, and ending before $I$, in such a way that it is admissible for computing the supremum for the maximal function. The average over this interval $J$ will then be larger than $Mf(y)$, because the average on $I$ was $Mf(y)$, and the values of the function in $I\setminus J$ are smaller. This leads to a contradiction, as it implies $Mf(y)> Mf(x)$ would not hold.

    When the value of the maximal function value comes from the limit of the averages over intervals with length going to zero, there will be an interval where the average is greater than $ Mf(y) -\delta$. This means that at least one point in this interval has $f$ taking a value above $ Mf(y) -\delta$. Choose that point for $w$.

    To select $z$, we use Lemma \ref{smaller} on the interval $(x-\epsilon , w)$, and pick $z$ as the $a$ given by the Lemma. Since $f(w)\geq Mf(y)-\delta$ and $f(z) \leq Mf(x)$, we are done.

\end{proof}
\begin{remark}
    Note that in the discrete setting, the $\epsilon$ and the $\delta$ can be ignored, as we can simply choose $z=x$. The case where the radius is zero reduces to $Mf(y)=f(y)$, so we can select $w=y$.
\end{remark}

%% file: disc.tex
In this section, we will prove a discrete version of the result in \cite{JPR}.

When defining the discrete version of the Nontangential Maximal Operator, we must specify the allowable deviation for the intervals from which we take the supremum, relative to the point at which we are computing the Maximal Function. We will use Definition \ref{H-L}

Before proceeding to the discrete version of the theorem, we should make some remarks about this definition.

\begin{remark}
    To achieve a result similar to the continuous case, we need to use the ceiling function in the definition, or else Lemma \ref{dM=f} would be false for $\alpha \in [ \frac 1 3 , \frac 1 2 )$. To construct a counterexample, consider the characteristic function of the set $\{ 0,4 \}$, which has a local maximum at $2$.
\end{remark}

\begin{remark}
    In the definition utilizing the ceiling function, the discrete version of Theorem \ref{diangle}, established down to the endpoint $\alpha = \frac 1 3$, also implies the same theorem with a definition excluding the ceiling, but only down to $\alpha = \frac 1 2$. This makes Lemma \ref{dM=f} stronger than any equivalent lemma defined without ceilings, showing that this strategy achieves the best possible result within these two definitions. 
\end{remark}

We are now ready to start the proof of Theorem \ref{diangle}.

\begin{proof}[Proof of Theorem \ref{diangle}]

This proof will be done for the endpoint case $\alpha = \frac 13$. For any bigger value of $\alpha$, we would be able to make the same choices as in this case, so the theorem will still hold. 

Fix any two integers $a$ and $b$, and the integer interval $I=[a..b]$. We aim to control the variation of the maximal function inside $I$, by the variation of $f$. If successful, this will yield the desired result. Consider the full set $I$ as a partition.

Inside $I$, either the maximal function is monotonous, or we can use Proposition \ref{divide_int} to divide our variation in up to three parts, the monotonous parts at the start and end, plus a system of peaks $\Pe = \{ p_1 < r_1 < q_1 = p_2 < \cdots < r_n < q_n \}$ in between. 

In the case where the maximal function is monotonous in the interval, we only need to control the variation $|\M^{\alpha}f(a)-\M^{\alpha}f(b)| $. This can be contolled by the variation of $f$ using Proposition \ref{extremos}.

When the maximal function is not monotonous, we can contol the variation of the (potentially empty) first and last parts using the Proposition \ref{extremos}. It gives us a control of those parts by the variation of the function outside $\Pe$. The only thing missing is then to contol the variation of the maximal function over the system of peaks $\Pe$ by the variation of the function over the same interval.

\begin{lemma}\label{dM>f}
Let $f \in BV(\Z).$ Then $\M^{\alpha}f(x) \ge f(x)$.
\end{lemma}

\begin{proof}
   
    This statement follows directly from the definition of $\M^{\alpha} f(x)$, as $f(x)$ corresponds to the average when $j=0$.
    
\end{proof}

\begin{lemma}\label{dM=f}
Let $f \in BV(\Z).$ Then, for each local maximum in our system of peaks at $m=r_i$, $\M^{\alpha}f(m) $, there exists another local maximizer $m'$ of $\M^{\alpha}f$ such that $\M^{\alpha}f(m) = \M^{\alpha}f(m') = f(m')$, and $\M^{\alpha}f(m)= \M^{\alpha}f(i)$ for all $i$ between $m$ and $m'$.
\end{lemma}

\begin{proof}

For the reason refered to in \ref{inf=min}, we know that the smallest radius used to compute the maximal function on a maximizer is finite. Thus, for a certain radius $R$ and center $c$ close enough to $m$, down to $|c-m|\leq \lceil \frac 13 R \rceil$.
$$
K = \M^{\alpha}f(m) = \frac 1{2R+1} \sum_{i=c-R}^{c+R} f(i).
$$

Knowing that $K$ is the maximum of the maximal function inside the interval containing the peak, we take that:
$$
\forall y\in [p_i..q_i], \M^{\alpha} f(y)\leq K.
$$

In the case where $R=0$, we are done as $c=m$ and $\M^{\alpha} f(m)= f(m)$. Assume, looking for a contradiction, that $R>0$ and no such $m'$ exists.

First, note that $\forall y \in [\lfloor c-\frac R3 \rfloor , \lceil c+\frac R3 \rceil ]$, we have that $\M^{\alpha}f(y) \geq K $, because the same average as the one for $m$ can be applied for any $y$ within this range. Since neither $\M^{\alpha} f(p_i)\geq K$ nor $\M^{\alpha} f(q_i)\geq K$, it follows that all such $y$ must be contained in $ [p_i..q_i] $. We can now conclude that $\M^{\alpha}f(y) = K$ throughout the interval. 

If, for any of those $y$, $f(y) = K$, we are done, as we can simply set $m' = y$. 

Otherwise, consider the points $c-R+\lfloor \frac {R-1}2 \rfloor$ and $c+R-\lfloor \frac {R-1}2 \rfloor$, and let us examine the intervals centered at those points with radius $\lfloor \frac {R-1}2 \rfloor$. Observe that these intervals do not overlap, and leave a gap between them which contains the point $c$. One of the averages over the two intervals must be larger than $K$. This is true because the average over the full interval $[c-R.. c+R]$ is $K$, and, by hypothesis, the function values at points outside both intervals is smaller than $K$. Without loss of generality, assume it the left interval satisfies this condition. 

Our goal now is to show that the average of the function over the left interval provides a lower bound for $\M^{\alpha}f(\lfloor c-\frac R3 \rfloor-1)$. To achieve this, we need to show that the argument is less than or equal to $c-R+ \lfloor \frac {R-1}2 \rfloor + \lceil \frac 13 \lfloor  \frac {R-1}2 \rfloor\rceil$.

$$
\lfloor c-\frac R3 \rfloor-1 \leq c-R+ \lfloor \frac {R-1}2 \rfloor + \lceil \frac 13 \lfloor  \frac {R-1}2 \rfloor\rceil \Leftrightarrow \\
R \leq \lfloor \frac {R-1}2 \rfloor + \lceil \frac R3 \rceil + \lceil \frac 13 \lfloor  \frac {R-1}2 \rfloor\rceil + 1
$$

We can prove this by induction. If it holds for $R=n$, it also holds for $R= n+6$, since multiples of $6$ jump out of the floors and ceilings as integers. Verifying this for the first six integers is straightforward, so the proof is complete.

Now, we know that $\M^{\alpha}f(\lfloor c-\frac r3 \rfloor-1) > K$, which leads to a contradiction with $K$ being a maximum of the maximal function. Between $p_i$ and $r_i$, the maximal function cannot exceed $K$, and between $p_i$ and $q_i$, $r_i$ is a maximizer of the maximal function within the interval. This contradiction arises from our assumption that $R>0$ and $m'$ not existing. We conclude that $R=0$ or $\M^{\alpha}f(m) = f(m')$.

\end{proof}

We can then control the variation of the maximal function inside each peak $\p=\{ p < r < q \} \subset \Pe$ by the variation of the function within the interval containing the peak. By Lemma \ref{dM>f}, we know that $f(p)\leq \M^{\alpha}f(p)$, and $f(q)\leq \M^{\alpha}f(q)$. Additionally, Lemma \ref{dM=f} guarantees the existence of a point $m'\in (p,q)$ such that $f(m')\geq \M^{\alpha}f(r)$. 
To confirm that $m'$ lies within the interval $(p,q)$, observe that between $m'$ and $r$, the maximal function takes the same value as in $r$, meaning neither $p$ nor $q$ lie between $m'$ and $r$. Since $r\in (p,q)$, it follows that $m'\in (p,q)$. 

Finally note that:
$$
\var_\Pe(\M^{\alpha}f) = 2\M^{\alpha}f(r) - \M^{\alpha}f(p) - \M^{\alpha}f(q) \leq 2f(m')-f(p)-f(q) \leq \var_{[p..q]}(f).
$$

\end{proof}

%% file: general.tex
In this section, we will provide an alternative proof of the theorem from \cite{JPR} in the continuous setting.

\begin{theorem}\label{angle} Fix $f \in BV(\R)$. For every $\alpha \in [\frac{1}{3},+\infty),$ we have that

\begin{equation}\label{mainteo}
\var(M^{\alpha}f) \le \var(f).
\end{equation}

\end{theorem}

\begin{proof}

Let $S=\{ x_1<x_2<...<x_N \}$ be a finite ordered subset of $\R$. We can then say that, by definition of total variation, $\var(M^{\alpha}f)$ is the supremum over all such subsets of the variation of $M^{\alpha}f$ between consecutive points in the sequence. For this reason, if we can bound the total variation of $\ma f$ over any such sequence by $\var (f)$, we achieve the desired inequality.

To facilitate the notation in this proof, we identify the last point in one peak with the first of the next one. This meansthat, the $q_i$,as it was defined in the original definition of peaks, is going to stay implied as equal to $p_{i+1}$ here. 

Using Proposition \ref{divide_int}, we can divide $S$ into up to three parts: a beginning, where the maximal function decreases from $x_1$ to $x_k$; an intermediate system of peaks $ \Pe = \{p_0 < r_0 < p_1 < ... < r_n < p_{n+1} \}$, where $p_0=x_k$, and $p_{n+1}=x_j$; and an increasing part at the end between $x_j$ and $x_N$. 

To control the variation inside the system of peaks, we can use the following Lemma:

\begin{lemma}
\label{mainclaim}
Given a system of peaks $\Pe = \{ p_0 < r_0 < p_1 < r_1 < ... < r_n < p_{n+1} \}$ for $\ma f$, and a $\delta >0$, we can find a system of peaks for $f$, $\Pe '= \{ a_0 < b_0 < a_1 < b_1 < ... < b_n < a_{n+1} \}$ such that $\var_\Pe (\ma f) \leq \var_{\Pe '} (f)+\delta$.

\end{lemma}

To prove this lemma, we will describe how to choose each $a_i$ and $b_i$. The idea is to choose them in a way such that
$$
\ma f(p_i) \geq f(a_i)\  \text{and}\  \ma f(r_i) \leq f(b_i)+\frac \delta n.
$$

Using \ref{smaller}, we can easily choose each $a_i$ as an $a$ given in the lemma. If we require the choice of the values for the $b_i$'s to satisfy $b_{i-1} < p_i < b_i$, then the open interval between both points will give us a possible $U$.

The proof is therefore reduced to choosing each $b_i$. Start by fixing an $i$, and define $r = r_i$. We will proceed to find a $b$ for the corresponding $r$ such that $b\in (p_i , p_{i+1} )$.

By \ref{casos_maximal}, the maximal function on $r$ cannot be attained for $t\rightarrow \infty$, as it has a value larger than the neighbouring points. The search for $b$ is therefore reduced for when either the value for the maximal function is given by fixed $t$, and $y$, or as the limit superior when $t\rightarrow 0$.

Now, look at the case $\ma f(r) = \limsup_{(t,y)\rightarrow (0,r) } \frac{1}{2t} \int_{y-t}^{y+t} |f(s)| \d s$. Because we are taking the limit when $t$ is going to zero, we can take a $(t,y)$ such that $(y-t, y+t) \subset (p_i,p_{i+1})$ and 
$$
\frac{1}{2t} \int_{y-t}^{y+t} |f(s)| \d s \geq \ma f(r)-\frac \delta n .
$$
Then, we can choose $b \in (y-t, y+t)$ such that $f(b) \geq \frac{1}{2t} \int_{y-t}^{y+t} |f(s)| \d s$. This solves the problem of choosing the $b$ in this case.

What is left now is to describe the choice of $b$ in the case where, for a particular width $t$ and a particular point $y$, we have:
$$
\ma f(r) = \frac{1}{2t} \int_{y-t}^{y+t} |f(s)| \d s.
$$

Let $A= [y-t, y-\frac 23 t]$, $B=[y-\frac 23 t, y-\frac 13 t]$, $C= [y-\frac 13 t, y]$, $D= [y, y+\frac 13 t]$, $E=[y+\frac 13 t, y+\frac 23 t]$ and $F= [y+\frac 23 t, y+t]$. 

The objective is to identify an interval $I \ni r$ from which to select a point $b$ such that the maximal function remains greater than $\ma f(r)-\frac \delta n$ inside $I$. Since $\ma f(p_i) < \ma f(r) > \ma f(p_{i+1})$, this setup ensures that both $p_i, p_{i+1}$ lie outside this interval. Thus, if we choose our $b$ within $I$, we only need to verify that $ f(b) \geq \ma f(r)-\frac \delta n$.

To proceed, we will consider 3 distinct cases, each involving a specific choice of $I$ and $b$. Before describing the cases, note that the average in one of the sets $A\cup B \cup C$ or $D \cup E \cup F$ must be at least $\ma f(r)$. This is true because the average of both averages gives us $ \ma f(r)$. Without loss of generality, we assume this set is $A\cup B \cup C$.

Case 1: The average of the function in $C$ or in $D$ is larger than $\ma f(r)$. 

In this case, our interval will be $ I=C\cup D$, and we choose $b$ equal to a point $p$ such that $f(p) \geq \ma f(r)$. 

The conditions above are verified: $r\in I$ is true by the definition of our maximal function with $\alpha \geq \frac 13$. Additionally, by the definition, the maximal function over $I$ is at least the value of the maximal function at $r$. The selection of the point $p$ relies on our case hypothesis and the fact that at least one point within $I$ where $f$ is at least the average over $I$. In this case, we are done.

Case 2: Both the averages of the function over $C$ and over $D$ are smaller than $\ma f(r)$. But, the average over $B$ is larger.

In this case, our interval is going to be $ I= B \cup C \cup D$. As in the previous case, we have that $r\in I$. The maximal function is above or equal to $\ma f(r)$ in $I$. For $C \cup D$, we have already verified in the previous case, and in $B$ it must be, as the average of the function over $A\cup B \cup C$ is larger than $\ma f(r)$, and this average is a lower bound for the maximal function at any point in $B$. Assuming that the average over $B$ is larger than $\ma f(r)$ lets us choose a point $p\in B \subset I$ such that $f(p) \geq \ma f(r)$.

Case 3: Both the averages of the function over $C$ and $B$ are smaller than $\ma f(r)$.

To construct an appropriate interval, we analyze the values of the function within $A$. Define $A_1$ as the second half of $A$. If there exists a point in $A_1$ where $f \geq \ma f(r)$, set $I=A_1 \cup B \cup C \cup D$. Otherwise, consider $A_2$ as the second half of $A\setminus A_1$. If a point in $A_2$ satisfies $f \geq \ma f(r)$, then set $I=A_2 \cup A_1 \cup B \cup C \cup D$. Continue halving the remaining interval and repeating this process until finding the smallest $k$ such that a point in $A_k$ satisfies $f \geq \ma f(r)$. Then, set $I=A_k \cup \dots \cup A_1 \cup B \cup C \cup D$.
This $k$ must exist and is finite because there must exist a point in $A$ where $f$ exceeds its average over $A$.
As in the other cases, $r\in I$. By construction, the maximal function over $I$ is larger than or equal to $\ma f(r)$. We can then choose the point $p$ in $A_k$ where $f(p) \geq \ma f(r)$.

Repeat this process for each $b_i$, and then select each $a_i$ within $(b_{i-1},b_i)$ using Lemma \ref{smaller}, with exceptions for the first and last points, which can can choosen from $(-\infty,b_1)$ and $(b_n,\infty)$, respectively.

\end{proof}

Thus, we have proven Lemma \ref{mainclaim} for $\delta>0$. Next, we aim to use Proposition \ref{extremos} to control the variation of $\ma f$ over the points in $S$ which are before, and after the system of peaks. By Proposition \ref{divide_int}, there are three different possible scenarios:

\begin{enumerate}
    \item The partition is given only by the system of peaks. The case $k=1$ and $j=N$
    In this case, we can use Lemma \ref{mainclaim} directly and find that:
    $$
    \sum_{i=1}^{N-1} |M^{\alpha}f(x_{i+1})-M^{\alpha}f(x_i)| \leq \var_\mathbbm{P}(M^{\alpha}f) \leq \var f + \delta '.
    $$
    Note that applying the Lemma directly yields $\var f + N\delta$, where $N$ denotes the number of peaks in the system. Since $N$ is finite,  $\forall \delta '>0$, choose a sufficiently small $\delta$ such that $N\delta < \delta '$.

    \item The partition includes the system of peaks and at least one point outside it. We will demonstrate the argument for the case where $1 < k < j < N$. The cases where $k=1$, or $j=N$, can be handled with analogous reasoning. 
    In this case, we have that:
    $$
    \sum_{i=1}^{N-1} |M^{\alpha}f(x_{i+1})-M^{\alpha}f(x_i)| \leq M^{\alpha}f(x_1)-M^{\alpha}f(x_k)+ \var_\mathbbm{P}(M^{\alpha}f) + M^{\alpha}f(x_N)-M^{\alpha}f(x_j).
    $$
    We would like to have a partition $S'= \{ y_1<y_2<...<y_{2n+3} \}$, such that:
    $$
    M^{\alpha}f(x_1)-M^{\alpha}f(x_k)+ \var_\mathbbm{P}(M^{\alpha}f) + M^{\alpha}f(x_N)-M^{\alpha}f(x_j) \leq \sum_{i=1}^{2n+3} |f(y_{i+1})-f(y_i)|.
    $$
    Using Lemma \ref{mainclaim}, we can control $ \var_\mathbbm{P}(M^{\alpha}f) \leq \var_{\Pe '}(f)$. To construct $\Pe '$ assign the middle points of $S'$ as follows: set $y_2=a_1$, $y_3=b_1$, and continue the pattern until $y_{2n+2}=a_{n+2}$. 
    Next, we address the control of $ M^{\alpha}f(x_1)-M^{\alpha}f(x_k) $ and $M^{\alpha}f(x_N)-M^{\alpha}f(x_j)$.
    For $ M^{\alpha}f(x_1)-M^{\alpha}f(x_k) $, fix any $\delta '$ as before. Choose $\epsilon >0$ such that $a_1+\epsilon < b_1 =y_3$. Using Lemma \ref{extremos} select $ w < z < y_3$ such that $ f(w)\geq M^{\alpha}f(x_1)-\delta '$, and $f(z) \leq M^{\alpha}f(x_k)$. assign $y_1 = w$ and redefine $y_2 = z$. 
    A similar argument can be applied using Lemma \ref{extremos} to define $y_{2n+2}$ and $y_{2n+3}$. This allows us to use the variation of $f$ over the partition given by the $y_i$'s to bound the variation of $M^{\alpha}f$ over $S$ minus three times $\delta '$.
    This will work because we have:
    \begin{equation} 
        \begin{split}
            \var_\mathbbm{P}(M^{\alpha}f)  & \leq \sum_{i=2}^{2n+2} |f(y_{i+1})-f(y_i)| +\delta ', \\
            M^{\alpha}f(x_1)-M^{\alpha}f(x_k) &  \leq f(y_1)-f(y_2)+\delta ', \\
            M^{\alpha}f(x_N)-M^{\alpha}f(x_j) &\leq f(y_{2n+3})- f(y_{2n+2})+\delta '.
        \end{split}
    \end{equation}

    \item The partition has no system of peaks. This corresponds to the case $k\geq j$
    In this scenario, the maximal function on the set $S$ is either monotonic, or it has a single turning point: decreases up to a point, and increases starting at that point. 
    The first case is solved with a direct application of Proposition \ref{extremos} with any $\delta >0$ and any $\epsilon >0$. 
    In the second case, where the maximal function decreases and then increases, we apply Proposition \ref{extremos} twice: once to control $M^{\alpha}f(x_1)-M^{\alpha}f(x_j)$, and again to control $M^{\alpha}f(x_N)-M^{\alpha}f(x_k)$. Then pick $y_1 < y_2 \leq y_3 < y_4$ as the points given by both instances of the Lemma, for any $\delta>0$, and for $\epsilon$ small enough such that $x_j+\epsilon < x_k -\epsilon$, or any value if they coincide. 
    This construction leads to:
    \begin{equation} 
        \begin{split}
            M^{\alpha}f(x_1)-M^{\alpha}f(x_j) &  \leq f(y_1)-f(y_2)+\delta ', \\
            M^{\alpha}f(x_N)-M^{\alpha}f(x_k) & \leq f(y_{4})- f(y_{3})+\delta '.
        \end{split}
    \end{equation}
    Now, $\forall \delta >0$, we can choose $\delta '$ to be as small as needed in order for it to be smaller than one third of $\delta$.

\end{enumerate}

Now, we have proven that, for any $\delta > 0$, $\var_S(M^{\alpha}f) \le \var(f) + \delta$, we can directly conclude that $\var_S(M^{\alpha}f) \le \var(f)$. 
If the inequality holds for any partition $S$, it must also hold for the total variation of $M^{\alpha}f$, and we get \ref{mainteo}.

%% file: main.bbl
\begin{thebibliography}{99}

\bibitem{AL}
J. M. Aldaz and J. Pérez Lázaro,
\newblock \textit{Functions of bounded variation, the derivative of the one dimensional maximal function, and applications to inequalities},
\newblock  Trans. Amer. Math. Soc. 359 (2007), no. 5, 2443–2461.

\bibitem{AP} 
J. M. Aldaz and J. P\'{e}rez L\'{a}zaro, 
\newblock \textit{Functions of bounded variation, the derivative of the one dimensional maximal function, and applications to inequalities},
 \newblock Trans. Amer. Math. Soc. 359 (2007), no. 5, 2443--2461.

\bibitem{BW} 
C. Bilz and J. Weigt,  
\newblock \textit{The one-dimensional centred maximal function diminishes the variation of indicator functions},
\newblock preprint (2021), arxiv.org/abs/2107.12404.


\bibitem{BCHP}
J. Bober, E. Carneiro, K. Hughes, and L. B. Pierce,
\newblock \textit{On a discrete version of Tanaka's theorem for maximal functions},
\newblock Proc. Amer. Math. Soc. 140 (2012), 1669-1680.


\bibitem{Ca}
E. Carneiro,
\newblock \textit{Regularity of maximal operators: recent progress and some open problems},
\newblock New Trends in Applied Harmonic Analysis, vol. 2 (2019), pp. 69--92.


\bibitem{CGRM} 
E. Carneiro, C. Gonz{\'a}lez-Riquelme, and J. Madrid,  
\newblock \textit{Sunrise strategy for the continuity of maximal operators},
\newblock J.~Anal. Math.~{\bf 148} (2022), 37--84.


\bibitem{CH}
E. Carneiro and K. Hughes,
\newblock \textit{On the endpoint regularity of discrete maximal operators}, 
\newblock Math. Res. Lett. 19, no. 6 (2012), 1245--1262.

\bibitem{CM}
E. Carneiro and D. Moreira, 
\newblock \textit{On the regularity of maximal operators},
\newblock Proc. Amer. Math. Soc. 136 (2008), no. 12, 4395--4404.

\bibitem{CMa}
E. Carneiro and J. Madrid, 
\newblock \textit{Derivative bounds for fractional maximal functions},
\newblock Trans. Amer. Math. Soc. 369 (2017), no. 6, 4063--4092.

\bibitem{CMP}
E. Carneiro, J. Madrid and L. B. Pierce,
\newblock \textit{Endpoint Sobolev and BV continuity for maximal operators},
\newblock J. Funct. Anal. 273 (2017), 3262--3294.


\bibitem{GR} 
C. Gonz{\'a}lez-Riquelme,  
\newblock \textit{Continuity for the one-dimensional centered Hardy-Littlewood maximal operator at the derivative level},
\newblock J.~Funct.~Anal.~{\bf 285} (2023), 110097.

\bibitem{GRK} 
C. Gonz{\'a}lez-Riquelme and D. Kosz,  
\newblock \textit{BV continuity for the uncentered Hardy--Littlewood maximal operator},
\newblock J.~Funct. Anal.~{\bf 281} (2021), 109037.

\bibitem{HO}
P. Hajlasz and J. Onninen,
\newblock \textit{On boundedness of maximal functions in Sobolev spaces},
\newblock Ann. Acad. Sci. Fenn. Math. 29 (2004), no. 1, 167--176.

\bibitem{Ki}
J. Kinnunen, 
\newblock \textit{The Hardy-Littlewood maximal function of a Sobolev function},
\newblock Israel J. Math. 100 (1997), 117--124.

\bibitem{KiSa}
J. Kinnunen and E. Saksman,
\newblock \textit{Regularity of the fractional maximal function},
\newblock Bull. London Math. Soc. 35 (2003), no. 4, 529--535. 

\bibitem{Ko}
S. Korry,
\newblock \textit{A class of bounded operators on Sobolev spaces},
\newblock Arch. Math. (Basel) 82 (2004), no. 1, 40--50. 

\bibitem{Ku}
O. Kurka,
\newblock \textit{On the variation of the Hardy-Littlewood maximal function},
\newblock Ann. Acad. Sci. Fenn. Math. 40 (2015), 109--133.

\bibitem{LW}
F. Liu and H. Wu,
\newblock \textit{On the regularity of the multisublinear maximal functions},
\newblock Canad. Math. Bull. 58 (4) (2015), 808--817.

\bibitem{M}
J. Madrid, 
\newblock \textit{Sharp inequalities for the variation of the discrete maximal function}, 
\newblock Bull. Aust. Math. Soc. 95 (2017), no. 1, 94--107.

\bibitem{PPSS}
C. P\'{e}rez, T. Picon, O. Saari and M. Sousa,
\newblock \textit{Regularity of maximal functions on Hardy-Sobolev spaces},
\newblock Bull. Lond. Math. Soc. 50 (2018), no. 6, 1007--1015.

\bibitem{JPR}
J. Ramos,
\newblock \textit{Sharp total variation results for maximal functions},
\newblock Ann. Acad. Sci. Fenn. Math. 44 (2019) 41-64

\bibitem{Ta}
H. Tanaka,
\newblock \textit{A remark on the derivative of the one-dimensional Hardy-Littlewood maximal function},
\newblock Bull. Austral. Math. Soc. 65 (2002), no. 2, 253–258.

\bibitem{Te}
F. Temur,
\newblock \textit{On regularity of the discrete Hardy-Littlewood maximal function},
\newblock preprint at http://arxiv.org/abs/1303.3993.







\end{thebibliography}
